\documentclass[11pt]{article}

\usepackage{fullpage}
\usepackage{amsthm}
\usepackage{amssymb}
\usepackage{amsmath}
\usepackage{graphicx}
\usepackage{tikz}
\usepackage{mathrsfs}
\usepackage{float}
\usepackage{makecell}
\usepackage{amscd}
\usepackage{multirow}
\usepackage{cite}
\usepackage{verbatim}
\usepackage{hhline}
\usepackage{comment}

\newcommand{\E}{\mathbb{E}}

\newcommand{\N}{\mathbb{N}}
\newcommand{\R}{\mathbb{R}}

\renewcommand{\epsilon}{\varepsilon}

\DeclareMathOperator{\alphavector}{\overline{\alpha}}
\usepackage{relsize}
\newtheorem{theorem}{Theorem}
\newtheorem{lemma}[theorem]{Lemma}

\newtheorem{corollary}[theorem]{Corollary}
\newtheorem{claim}[theorem]{Claim}
\theoremstyle{definition}

\newtheorem{definition}[theorem]{Definition}

\title{Linear embeddings of random complexes}
\author{Andrew Newman\thanks{Carnegie Mellon University}}

\begin{document}
\maketitle
\begin{abstract}
For $X \sim X(n; 1, n^{-\alpha_1}, n^{-\alpha_2}, ...)$ in the multiparameter random simplicial complex model we establish necessary and sufficient strict inequalities on the $\alpha_i$'s to linearly embed the complex into $\R^{2d}$.
\end{abstract}
\section{Introduction}
Recall the multiparameter random simplicial complex model introduced by Costa and Farber \cite{CostaFarberMultiparameter}. Given $n \in \N$ and a sequence of probabilities $\textbf{p} = (p_i)_{i \geq 0}$ we sample $X \sim X(n; \textbf{p})$ starting with the ground set $[n]$. We first include each element of the ground set as a vertex independently with probability $p_0$, from here we include each edge between exisiting vertices with probability $p_1$, and then each triangle whose boundary is included with probability $p_2$, and so on. This model interpolates between the clique complex model, introduced first by Kahle \cite{KahleRandomClique}, which is the case where $p_1$ can vary but all other $p_i$'s are 1 and the Linial--Meshulam--Wallach $d$-dimensional model \cite{LM, MW} where $p_d$ can vary, $p_i = 1$ for $i < d$ and $p_i = 0$ for $i > d$ (note that the $d = 1$ case for Linial--Meshulam--Wallach is the Erd\H{o}s--R\'enyi random graph). 

Here we establish thresholds for linear embeddings of random simplicial complexes in Euclidean space. Recall that a \emph{linear map} from a simplicial complex $\Delta$ into Euclidean space $\R^d$ is a map $\phi: \Delta \rightarrow \R^d$ that is determined linearly by the image of vertices. That is, the image of any face $\sigma \in \Delta$ under $\phi$ is the convex hull in $\R^d$ of the image of the vertices of $\sigma$. We say that $\Delta$ embeds linearly in $\R^d$ if there is an injective linear map $\phi: \Delta \hookrightarrow \R^d$.

Perhaps the most well known situation concerning linear embeddings of simplicial complexes is graph planarity. A graph $G$ is planar if and only if it admits a linear embedding into $\R^2$. We could take this as a definition of planar graphs, but it turns out that when it comes to planar graphs it doesn't matter if we characterize them by existence of linear embeddings or existence of \emph{topological embeddings}. The result that a graph is planar if and only if it can be drawn in the plane with all edges as straight-line segments is known as F\'ary's theorem. The analogous statement for higher dimensions was first demonstrated to be false by Brehm \cite{BrehmMobiusStrip} who gave a triangulation of the M\"obius strip on nine vertices that does not linearly embed in $\mathbb{R}^3$. Results and discussion in \cite{FHSS} examine certain cases of equivalence or nonequivalence of topological and linear embeddings.

In between linear embeddings and topological embeddings are PL-embeddings (piecewise-linear embeddings). A simplicial complex is PL-embeddable in $\R^d$ if it admits a subdivision that is linearly embeddable in $\R^d$. It is known that PL-emdeddability is strictly weaker than linear embeddability in general. Brehm and Sarkaria \cite{BrehmSarkaria} disproved a conjecture of Gr\"unbaum that PL-embeddability of a $d$-complex in $2d$-dimensional space implies linear embeddability. A result of Matou\v{s}ek, Tancer, and Wagner \cite{MTW} establishing complexity results on PL-embeddings of simplicial complexes gives another proof that PL-embeddability does not imply linear embeddability by demonstrating that the two decision problems are not in the same complexity class.  Here as a corollary to the main result we give a proof of this same nonequivalence via the probabilistic method by establishing threshold results for linear embeddability of random complexes and contrasting them with known threshold results for PL-embeddability, see the concluding remarks for a precise statement on this.

 We consider multiparameter random complexes in a sparse regime where each $p_i = n^{-\alpha_i}$ for some fixed $\alpha_i \in [0, \infty]$ (with $\alpha_i = \infty$ corresponding to $p_i = 0$).  Furthermore, we assume that $p_0 = 1$. This is a convenience for the proofs, but by a straightforward but perhaps somewhat tedious application of large deviation inequalities on the number of vertices the natural extension of the main result also holds when $p_0$ is allowed to depend on $n$ provided that $p_0 n \rightarrow \infty$ as $n \rightarrow \infty$.

To state our main result we first introduce some notation. For $\alphavector = (\alpha_1, \alpha_2, ...)$ we denote $X(n; 1, n^{-\alpha_1}, n^{-\alpha_2}...)$ by $X(n; n^{-\alphavector})$. For each $s \geq 1$ we let $v_s$ denote the vector of binomial coefficients  ${v}_s = \left( \binom{s}{i + 1} \right)_{i \geq 1}$. By linearity of expectation the expected number of $(s - 1)$-simplices in $X(n, n^{-\alphavector})$ is 
\[\Theta(n^{s - \alphavector \cdot v_s}).\]
Our main result establishes a threshold in $\alphavector$ for embeddability in even dimensions.
\begin{theorem}\label{maintheorem}
Fix $d \geq 1$ we have the following for random simplicial complexes in the multiparameter model $X \sim X(n; n^{-\alphavector})$. If $d + 1 - \alphavector \cdot v_{d + 1} < 1$ then with high probability $X$ linearly embeds in $\R^{2d}$ while if $d + 1 - \alphavector \cdot v_{d + 1} > 1$ with high probability $X$ does not linearly embed in $\R^{2d}$. 
\end{theorem}
Stated another way, our main result says if the $\alphavector$ is selected so that the expected number of $d$-faces is at most $n^{1 - \epsilon}$ for some $\epsilon > 0$ then the random complex will embed in $\R^{2d}$ with high probability, while an average number of $d$-faces of order at least $n^{1 + \epsilon}$ implies that the random complex will not embed in $\R^{2d}$. Somewhat surprisingly then the threshold to embed a random complex in $\R^{2d}$ turns out to only depends on the $d$-skeleton of the random complex. 

\section{Preliminaries}
Recall that a $d$-complex $\Delta$ is alway embeddable in $\R^{2d + 1}$. A generic set of points in $\R^{m}$ is one where for any $k$ there is no affine subspace of $\R^m$ containing more than $k + 1$ of the points and for any two disjoint subsets $A$ and $B$ of the point set each of size at most $m$, the affine subspace of dimension $|A| - 1$ spanned by the points in $A$ and the affine subspace of dimension $|B| - 1$ spanned by the points in $B$ intersect as an affine subspace of dimension $|A| + |B| - 2 - m$, if this quantity is nonnegative and don't intersect at all otherwise. If we place the points of a $d$-complex $\Delta$ in $\R^{2d + 1}$ generically then it will be an embedding as two $d$-dimensional subspaces generically don't intersect in $\R^{2d + 1}$.

Observe that any set of points in $\R^{m}$ are moved to generic position by an arbitrarily small perturbation. The proof of the sparse side of Theorem \ref{maintheorem} will give a generic embedding of our random complex in $\R^{2d}$, and we will rule out the existence of a generic embedding on the dense side of the claimed threshold. If there is no generic embedding of our finite simplicial complex there is no embedding at all as a small perturbation of an embedding will still be an embedding.

The proof for the sparse side will follow a collapsibility argument. We will prove a linear embedding analogue of a result of Horvati\'c from \cite{Horvatic} about PL-embeddability. A face of a simplicial complex is said to be \emph{free} if it is properly contained in only one other face. The removal of a free face and the face that properly contains it is called an elementary collapse, and a complex is $d$-collapsible if there is a sequence of elementary collapse on the complex that leave behind only faces of dimension at most $d - 1$. Theorem 3.2 of \cite{Horvatic} implies that a $d$-complex which is $d$-collapsible PL embeds in $\R^{2d}$. We'll prove on the sparse side that if a stronger collapsibility property holds for a $d$-complex it will linearly embed in $\R^{2d}$ and then show that with high probability our random complex satisfies that strong collapsibility property.

For the dense side we make use of Radon's Theorem. Recall that Radon's Theorem states that for any set of $d + 2$ points in $\R^d$ there is a partition into two sets $A$ and $B$ so that the convex hull of $A$ intersects the convex hull of $B$. For the dense side we start with a generic embedding of the points into $\R^{2d}$ and then generate the complex and count how many sets of $2d + 2$ points, in expectation, have a Radon partition $(A, B)$ with the simplex on $A$ and the simplex on $B$ both included in the complex. We show that the expected number of such point sets is large and then prove a concentration of measure inequality that survives multiplication by the number of combinatorially distinct ways to place the $n$ points in $\R^{2d}$ (the number of order types).

Both for the proof of the sparse side and later for the proof of the concentration inequality on the dense side, we need the following lemma that gives basic facts about the $f$-vector of a random complex in the multiparameter model.

\begin{lemma}\label{fvectortheorem}
For any $s \geq 2$ if $s - \alphavector \cdot v_s < 1$ then $s + 1 - \alphavector \cdot v_{s + 1} < 0$, and for $s \geq 3$ if $s - \alphavector \cdot v_s > 1$ then $s - 1 - \alphavector \cdot v_{s - 1} > 1$.
\end{lemma}
\begin{proof}
Fix $\epsilon > 0$ and suppose that $s - \alphavector \cdot v_s = 1 -\epsilon$. We maximize $s + 1 - \alphavector \cdot v_{s + 1}$ subject to this constraint and the constraints that the $\alpha_i$'s are all nonnegative. Obviously if one of the $\alpha_i$'s is infinity so that $\overline{\alpha} \cdot v_{s + 1} = \infty$ there is nothing to prove, so we may assume all the relevant $\alpha_i$'s are finite. The region of the hyperplane determined by $s - \alphavector \cdot v_s =1 - \epsilon$ in the first orthant is a simplex on $s - 1$ vertices in $\R^{s - 1}$. By the simplex method we only have to check $s + 1 - \alphavector \cdot v_{s + 1}$ at the vertices of this simplex. Moreover we may assume that $\alpha_s = 0$. At the vertices of the feasible region exactly one of the $\alpha_i$'s is nonzero. Let $\alpha_{j - 1} \neq 0$ for some $j \in \{2, ..., s\}$ with the rest of the $\alpha_i$'s in this range being 0. So we have that $s - \alpha_{j - 1} \binom{s}{j} = 1 - \epsilon$.  Now
\begin{eqnarray*}
s + 1 - \alpha_{j - 1} \binom{s + 1}{j} &=& s + 1 - \alpha_{j  - 1} \left( \binom{s}{j} + \binom{s}{j - 1} \right) \\
&=& s - \alpha_{j - 1} \binom{s}{j} + 1 - \alpha_{j - 1} \binom{s}{j - 1} \\ 
&=& 1 - \epsilon + 1 - \alpha_{j - 1} \binom{s}{j - 1}  \\
&=& 1 - \epsilon + 1 - \alpha_{j - 1} \binom{s}{j} \left(\frac{\binom{s}{j - 1}}{\binom{s}{j}} \right) \\
&=& 2 - \epsilon - (s - 1 + \epsilon)  \frac{j}{s - j + 1} 
\end{eqnarray*}
Now we have to verify that this is negative, as $\frac{j}{s + 1 - j}$ is an increasing function and $j \geq 2$, we have
\begin{eqnarray*}
2 - \epsilon - (s - 1 + \epsilon)  \frac{j}{s - j + 1}  \leq 2 - \epsilon - (s - 1 + \epsilon) \frac{2}{s - 1}
\end{eqnarray*}
The righthand side is negative provided that
\[2 - \epsilon < 2 \left( \frac{s - 1 + \epsilon}{s - 1} \right)\]
And this holds clearly since the right hand side is bigger than 2.

For the other part, we minimize $s - 1 - \alphavector \cdot v_{s - 1}$ subject to $s - \alphavector \cdot v_{s} = 1 + \epsilon$. As before it suffices to check in the cases that a single $\alpha_{j - 1} \neq 0$ for some $j \in \{2, ..., s\}$ with the rest of the $\alpha_i$'s being zero and so $s - \alpha_{j - 1} \binom{s}{j} = 1 + \epsilon$. We have 
\begin{eqnarray*}
s - 1 - \alpha_{j - 1} \binom{s - 1}{j} &=& s - 1 - \alpha_{j  - 1} \left( \binom{s}{j} - \binom{s - 1}{j - 1} \right) \\
&=& s - 1 - \alpha_{j - 1} \binom{s}{j} + \alpha_{j - 1} \binom{s - 1}{j - 1} \\
&=&  - 1 + (1 + \epsilon) + \alpha_{j - 1} \binom{s}{j} \left( \frac{j}{s} \right) \\
&=& - 1 + (1 + \epsilon) + (s - 1 - \epsilon)\left(\frac{j}{s} \right)
\end{eqnarray*}
Again it suffices to check in the case that $j = 2$ as that's what makes final expression above the smallest. So we check
\[\epsilon + (s - 1 - \epsilon) \left( \frac{2}{s} \right ) > 1\]
However our assumption that $s - \alphavector \cdot v_{s} = 1 + \epsilon$ gives that $\epsilon$ is at most $s - 1$ and the above expression attains a minimum at $\epsilon = 0$ where we just have to check that 
\[2 \left(\frac{s - 1}{s}\right) > 1\]
and this holds for $s \geq 3$ as required.
\end{proof}

\section{Sparse side}
For the sparse side we will prove a certain collapsibility condition on our random complexes that implies the complex is embeddable. For $X$ a $d$-complex, we associate a $(d + 1)$-uniform hypergraph $\mathcal{H}$ whose hyperedges are the $d$-faces of $X$. If $\mathcal{H}$ has no 2-core then $X$ will embed in $\R^{2d}$. By 2-core we mean a subhypergraph of $\mathcal{H}$ in which every vertex is contained in at least two hyperedges. If $\mathcal{H}$ has no 2-core then there is a sequence of vertex deletions obtained by removing a vertex of degree at most 1, along with the hyperedge that contains it, that removes all the vertices. We will say the pure part of $X$ has no 2-core to mean that the associated hypergraph has no 2-core. Recall that the pure part of a $d$-complex is the subcomplex generated by the downward closure of the $d$-faces. The next theorem shows that the absences of a 2-core in dimension $d$ implies embeddability in dimension $2d$.
\begin{theorem}\label{CimpliesE}
For $d \geq 1$, if $X$ is a finite $d$-complex and the pure part of $X$ has no 2-core then $X$ linearly embeds into $\R^{2d}$.
\end{theorem}
\begin{proof}
If we can embed the pure part of $X$ into $\R^{2d}$ generically then this embedding can be extended to an embedding of the entire complex. Indeed if the vertices of $X$ are mapped into $\R^{2d}$ generically then the only faces of $X$ that might intersect one another in their interior are pairs of $d$-faces, each pair of which could possibly intersect at a single point. Therefore once we have placed the points into $\R^{2d}$ generically so that the pure part of $X$ embeds, the lower dimensional faces can be added without obstructing the embedding. 

Let $f_1, ..., f_k$ be an ordering of the $d$-faces of $X$ obtained by removing vertices of degree one (called \emph{free vertices}) with $f_k$ as the first $d$-face removed and $f_1$ as the last face removed. We show that if we have embedded $f_1 \cup \cdots \cup f_{\ell}$ then we may extend the embedding to include $f_{\ell + 1}$. Obviously $f_1$ as a $d$-simplex embeds into $\R^{2d}$. Now suppose we have embedded $f_1 \cup \cdots \cup f_{\ell}$. When $f_{\ell + 1}$ is added to $X$ it comes with a new vertex that belongs to $f_{\ell + 1}$ but not to any other $f_i$'s we've embedded so far by how the ordering is defined. Thus 
\[(f_1 \cup \cdots \cup f_{\ell}) \cap f_{\ell + 1} \]
contains at most $d$ vertices. If the intersection contains less than $d$ vertices then we place all but one of the new vertices of $f_{\ell + 1}$ in $\R^d$ generically with the rest of the embedding. Now we fill in the convex hull of these $d$ vertices of $f_{\ell + 1}$. By general position this will be compatible with the embedding as everything we are adding is of dimension at most $d - 1$. Now we want to add the last vertex and the simplex $f_{\ell + 1}$ attaching it along the $(d - 1)$-simplex on the first $d$-vertices. Call this $(d - 1)$-simplex $\sigma$. We take an arbitrary affine $d$-space through $\sigma$ in $\R^{2d}$, call this space $K$. Now $K$ will intersect the complex we have embedded so far only at $\sigma$ and at some finite number of points. Therefore there is some thickening of $\sigma$ within $K$ that avoids all the finitely many points of intersection with the rest of the complex. This can be obtained by thickening $\sigma$ by half the distance to the nearest of the finitely many points. Now put the new vertex $v$ inside this this thickening and take the convex of $v$ and $\sigma$ and we have extended the embedding to include $v \cup \sigma = f_{\ell + 1}$. Therefore inductively the pure part of $X$ can be embedded in $\R^{2d}$ generically and this enough to conclude that all of $X$ can be embedded in $\R^{2d}$.
\end{proof}

By Theorem \ref{CimpliesE} we want to prove a collapsibility result about $X \sim X(n; n^{-\alphavector})$. This is a different type of collapsibility than is studied in \cite{MalenCollapsibility} and \cite{NewmanOneSided}. Here we are collapsing (vertex, $d$-simplex)-pairs and these other papers deal with collapsing ($(d - 1)$-simplex, $d$-simplex)-pairs. The overall strategy though is the very similar: we prove that there are no large obstructions to collapsibility, and that local obstructions to collapsibility are too dense to appear in this sparse regime. We start with a definition.

\begin{definition}
For $X$ a $d$-complex, associate a graph $G(X)$ whose vertices are the $d$-simplices of $X$ with $\sigma$ adjacent to $\tau$ exactly when $\sigma \cap \tau \neq \emptyset$. We say that a subcomplex $Y$ of $X$ is \emph{weakly connected} if for any $\sigma$ and $\tau$ in $Y$ there is a path within $G(X)$ between $\sigma$ and $\tau$. The weakly connected components in $X$ are the maximal weakly connected subcomplexes.
\end{definition}

We could also refer to weak connectivity as path connectivity but we use weak connectivity to contrast with strong connectivity from the other papers in the literature on collapsibility. Clearly a minimal 2-core of $X$ is weakly connected. We first rule out large weakly connected components as that will rule out large minimal 2-cores.

\begin{lemma}\label{LargeCoreLemma}
Fix $d \geq 1$ and $\epsilon > 0$, for vector $\alphavector$ so that $d + 1 - \alphavector \cdot v_{d + 1} = 1 - \epsilon$ there is some constant $L = L(d, \epsilon)$ so that with high probability $X \sim X(n; n^{-\alphavector})$ contains no weakly connected subcomplex on more than $L$ vertices. 
\end{lemma}
\begin{proof}
If $\epsilon > 1$ a simple first moment argument shows that there are no $d$-faces at all with high probability, so we assume $\epsilon \leq 1$. The idea is to show that for $L$ large, but constant independent of $n$, if $Y$ is any weakly connected subcomplex on $L$ vertices there is some $\delta > 0$ so that the dot product of the $f$-vector of $Y$ denoted $(L, f_Y)$, with $f_Y$ as the $f$-vector omitting the number of vertices, with the vector $(1, \alphavector) := (1, \alpha_1, ..., \alpha_d)$ satisfies
\[(L, f_Y) \cdot (1, \alphavector) = L - f_Y \cdot \alphavector < -\delta.\]
If we have this then the expected number of weakly connected subcomplexes with $L$ vertices in $X$ is at most
\[\binom{n}{L} 2^{2^L} n^{-\alphavector \cdot f_Y} \leq 2^{2^L} n^{L - \alphavector \cdot f_Y} \leq 2^{2^L} n^{-\delta} = o(1).\]
\noindent
The $2^{2^L}$ is simply a trivial upper bound on the number of simplicial complexes on $L$ vertices. 

Now we verify the claimed bound on $L - f_Y \cdot \alphavector$. Suppose that $Y$ is a weakly connected $d$-complex, then by taking a spanning tree on $G(Y)$ we can find an ordering $\sigma_1, ..., \sigma_K$ on the $d$-faces of $Y$ so that for each $2 \leq k \leq K$, $\sigma_{k} \cap (\sigma_1 \cup \cdots \cup \sigma_{k - 1})$ is nonempty. Let $\beta = d + 1 - \alphavector \cdot v_{d+1}$, this is the value of the dot product of the $f$-vector of $\sigma_1$ with $(1, \alphavector)$. Throughout the process of adding the $\sigma_i$'s $\beta$ will denote the value of $f_0 - f_1\alpha_1 - \cdots - f_d \alpha_d$ as the $f$-vector changes. We keep track of how $\beta$ changes as we add the $\sigma_i$'s in their given order. Specifically we want to argue that if $L$ is large enough, that $\beta$ eventually becomes negative. We immediately see that if we add a $\sigma_i$ that doesn't contain any vertices that aren't already in $\sigma_{1} \cup \cdots \cup \sigma_{i - 1}$ that $\beta$ cannot increase, so only have to consider how it changes when we add a face that includes new vertices. Suppose adding $\sigma_i$ adds $t$ new vertices to $Y$, then the effect adding $\sigma_i$ on $\beta$ is adding to it at most the quantity
\[\gamma(t) := t - \alpha_1 \left(\binom{t}{2} + g(t, d, 1)\right) - \alpha_2\left(\binom{t}{3} + g(t, d, 2)\right) - \cdots - \alpha_{d - 1} \left(\binom{t}{d} + g(t, d, d - 1)\right)\]
where $g(t,d,k)$ is the number of $k$-dimensional faces in the $d$-simplex containing at least one vertex from a fixed subset of size $t$ and at least one vertex outside that set, so 
\[g(t, d, k) = \binom{d + 1}{k + 1} - \binom{t}{k + 1} - \binom{(d + 1) - t}{k + 1}\]
So we maximize (over $\alphavector$, not over $t$)  $\gamma(t)$ subject to $\alphavector \geq 0$ and $d - \alphavector \cdot v_{d + 1} = -\epsilon$. Similar to the proof of Lemma \ref{fvectortheorem} it suffices to check when exactly one entry of $(\alpha_1, ..., \alpha_d)$ is nonzero. Set $j$ so that $\alpha_j \neq 0$ and $\alpha_i = 0$ for $i \in \{1, ..., d\}$, with $i \neq j$ and $d - \alpha_j \binom{d + 1}{j + 1} = -\epsilon$. In this case then for any $t \in \{1, ..., d\}$,
\begin{eqnarray*}
\gamma(t) &=& t - \alpha_j \left( \binom{t}{j + 1} + \binom{d + 1}{j + 1} - \binom{t}{j + 1} - \binom{(d + 1) - t}{j + 1} \right) \\
&=& t - \alpha_j \binom{d + 1}{j  + 1} + \alpha_j \binom{(d + 1) - t}{j + 1}  \\
%&=& t - (d + \epsilon) + \alpha_j \binom{d + 1}{j + 1} \left( \frac{\binom{d + 1 - t}{j + 1}}{\binom{d + 1}{j + 1}} \right) \\
%&=& t - (d + \epsilon) + \alpha_j \binom{d + 1}{j + 1} \left(\frac{(d - j) \cdots (d - t + 1)}{(d + 1) \cdots (d - t + 2)} \right)
\end{eqnarray*}
We want to verify that this is always negative and bounded away from zero. Since we've assumed that
\[\alpha_j = \frac{d + \epsilon}{\binom{d + 1}{j + 1}},\]
it suffices to prove that for all $t \in \{1, ..., d\}$ and $j \in \{1, ..., d\}$
\[\frac{t}{\binom{d + 1}{j + 1} - \binom{(d + 1) - t}{j + 1}} \leq \frac{d}{\binom{d + 1}{j + 1}}.\]

This is equivalent to showing that 
\[\frac{t}{d} \leq \frac{\binom{d + 1}{j + 1} - \binom{(d + 1) - t}{j + 1}}{\binom{d + 1}{j + 1}}.\]

We prove this as the following claim:
\begin{claim}
For any $n$, $t \in \{1, ..., n - 1\}$, and $k \in \{2, ..., n\}$
\[\frac{t}{n - 1} \leq \frac{\binom{n}{k} - \binom{n - t}{k}}{\binom{n}{k}}\]
\end{claim}
\begin{proof}
The righthand side of the inequality is the probability that a uniform random $k$-element subset of $\{1, ..., n\}$ contains at least one element of a fixed subset of $\{1, ..., n\}$ of size $t$. Obviously that probability gets larger as $k$ increases. So it suffices to verify the inequality when $k = 2$.
\begin{eqnarray*}
&&\frac{\binom{n}{2} - \binom{n - t}{2}}{\binom{n}{2}} \geq \frac{t}{n - 1} \\
&<=>& \frac{n(n - 1) - (n - t)(n - t - 1)}{n(n - 1)} \geq \frac{t}{n - 1} \\
&<=>& \frac{n^2 - n - (n^2 - nt - n - nt + t^2 + t)}{t} \geq n \\
&<=>& \frac{2nt - t^2 - t}{t} \geq n \\
&<=>& 2n - t - 1 \geq n \\
&<=>& n - 1 \geq t
\end{eqnarray*}
and we have assumed that $t \leq n -1$ so we have the claim
\end{proof}
It follows that for every $\sigma_i$ that includes a new vertex $\beta$ decreases by at least some fixed positive constant depending on $d$ and $\epsilon$, and $\beta$ never increases. Thus $\beta$ can be made arbitrarily negative by setting $L$ arbitrarily large. This is sufficient for the first moment argument covered at the beginning of the proof to hold.
\end{proof}

Next we show that small subcomplexes on at most $L$ vertices are too sparse to be 2-cores. To that end we prove the following lemma.
\begin{lemma}\label{SmallCoreLemma}
If $\alphavector$ satisfies $d + 1 - \alphavector \cdot v_{d + 1} < 1$ and $Y$ is a $d$-complex so that for every subcomplex $Y' \subseteq Y$, 
\[\langle f_0(Y'), f_1(Y'), ..., f_d(Y') \rangle \cdot \langle 1, -\alpha_1, ..., - \alpha_d \rangle > 0\]
then $Y$ has a vertex in at most one $d$-face, in particular $Y$ does not contain a 2-core.
\end{lemma}
\begin{proof}
Suppose $Y$ is a $d$-complex satisfying the inequality for $\alphavector$ so that $d + 1 - \alphavector \cdot v_{d + 1} < 1$. For a vertex $w$ let $\deg_i(w)$ denote the number of $i$-dimensional faces containing $w$, then 
\[\sum_{w \in Y} \left( 1 - \frac{1}{2} \deg_1(w) \alpha_1 - \frac{1}{3} \deg_2(w) \alpha_2 - \cdots - \frac{1}{d + 1} \deg_d(w) \alpha_d \right) = f_0(Y) - \alpha_1 f_1(Y) - \cdots - \alpha_d f_d(Y) > 0.\]

Therefore there is some $w$ so that 
\[\sum_{i = 1}^d \frac{1}{i + 1} \deg_i(w) \alpha_i < 1.\]

We want to show that $w$ does not belong to at least two $d$-simplices.  Suppose that $w$ is contained in at least two $d$-simplices, then the link of $w$ contains two $(d - 1)$-simplices $\sigma$ and $\tau$. If $|\sigma \cap \tau| = m \leq d - 1$ then

\[\sum_{i = 1}^d \frac{1}{i + 1} \deg_i(w) \alpha_ i\geq \sum_{i =1}^d \frac{1}{i + 1} \left( 2\binom{d}{i} - \binom{m}{i} \right) \alpha_i\]

The assumption that $d + 1 - \alphavector \cdot v_{d + 1} < 1$ gives us that

\[\sum_{i = 1}^{d} \binom{d+1}{i + 1} \alpha_i > d.\]
Combining the previous three inequalities then there must be some $0 \leq m \leq d$ so that
\[\sum_{i = 1}^{d} \binom{d+1}{i + 1} \alpha_i >  \sum_{i =1}^d \frac{d}{i + 1} \left( 2\binom{d}{i} - \binom{m}{i} \right) \alpha_i\]

We next prove the following claim. Doing so contradicts our assumption that $w$ belongs to at least two $d$-simplices, so $w$ can be taken to be our vertex of degree at most 1 and we then apply the same argument to $Y\setminus \{w\}$ and repeat until all the simplices have been deleted by removing free or isolated vertices verifying that $Y$ does not contain a 2-core.
\begin{claim}
For any $d \geq 1$, $i \in \{1, ..., d\}$ and $m \in \{0, ..., d - 1\}$, 
\[\binom{d + 1}{i + 1} \leq \frac{d}{i + 1} \left( 2\binom{d}{i} - \binom{m}{i} \right) \]
\end{claim}
\begin{proof}
Clearly we just have to check $m = d - 1$.
\begin{eqnarray*}
 &&\frac{d}{i + 1} \left( 2\binom{d}{i} - \binom{d - 1}{i} \right) - \left( \frac{d + 1}{i + 1} \right) \geq 0 \\
&<=>& d \left(2 \frac{d!}{(d - i)!} - \frac{(d - 1)!}{(d - 1 - i)!} \right) - \frac{(d + 1)!}{(d - i)!} \geq 0 \\
&<=>& d \left(2d! - (d - 1)! (d - i) \right) - (d + 1)! \geq 0 \\
&<=>& d! \left( 2d - (d - i) \right) \geq (d + 1)! \\
&<=>& d + i \geq d + 1 \\
&<=>& i \geq 1.
\end{eqnarray*}
The claim is verified.
\end{proof}
\end{proof}

With the previous lemmas we are ready to prove the sparse side of Theorem \ref{maintheorem}
\begin{proof}[Proof of sparse side of Theorem \ref{maintheorem}]
Suppose that $d + 1 - \alpha \cdot v_{d + 1} < 1$. We show that with high probability $X \sim X(n; n^{-\alphavector})$ is at most $d$-dimensional and the pure $d$-part of $X$ has no 2-core. If this holds then by Theorem \ref{CimpliesE}, $X$ will be embeddable in $\R^{2d}$. From Lemma \ref{fvectortheorem}, $d + 1 - \alpha \cdot v_{d + 1} < 1$ implies that $d + 2 - \alpha \cdot v_{d + 2} < 0$ and the expected number of $(d+1)$-faces is $O(n^{d + 2 - \alpha \cdot v_{d + 2}}) = o(1)$, so $X$ is $d$-dimensional with high probability. Additionally Lemma \ref{LargeCoreLemma} implies that with high probability there is some $L$ so that all minimal 2-cores in $X$ have at most $L$ vertices. 

However, the possibility of small cores is handled by Lemma \ref{SmallCoreLemma} and a first moment argument. By Lemma \ref{SmallCoreLemma} for any $Y$ that is a $d$-dimensional 2-core there is a subcomplex $Y' \subseteq Y$ so that 

\[\langle f_0(Y'), f_1(Y'), ..., f_d(Y') \rangle \cdot \langle 1, -\alpha_1, ..., - \alpha_d \rangle \leq 0\]
However if equality holds, then by some small perturbation of $\alphavector$ we could get a contradiction to Lemma \ref{SmallCoreLemma} and so we actually have a strict inequality. Now as we have ruled out minimal 2-cores on more than $L$ vertices if we can show that $X$ does not contain any subcomplexes $Y$ on at most $L$ vertices so that $\langle f_0(Y), f_1(Y), ..., f_d(Y) \rangle \cdot \langle 1, -\alpha_1, ..., - \alpha_d \rangle < 0$ we will be done. Let $\mathcal{F}$ be the family of such subcomplexes. As $\mathcal{F}$ is finite there is some $\delta > 0$ so that for all $Y \in \mathcal{F}$ $\langle f_0(Y), f_1(Y), ..., f_d(Y) \rangle \cdot \langle 1, -\alpha_1, ..., - \alpha_d \rangle < -\delta$. By linearity of expectation the expected number of complexes in $\mathcal{F}$ that are subcomplexes of $X$ is at most
\[|\mathcal{F}|L! n^{-\delta}\]
because the number of copies of $Y$ in $X$ for $Y \in \mathcal{F}$ in expectation is at most
\[\binom{n}{f_0(Y)} n^{-f_1(Y) \alpha_1 -f_2(Y)\alpha_2 - \cdots - f_d(Y) \alpha_d} < n^{-\delta}.\]
As $|\mathcal{F}|$ and $L$ are constants independent of $n$ the expected number of subcomplexes of $\mathcal{F}$ in $X$ is $o(1)$, so with high probability $X$ contains no element of $\mathcal{F}$ and hence no 2-core. Thus with high probability $X$ embeds in $\R^{2d}$ as it satisfies the assumptions of Theorem \ref{CimpliesE}.
\end{proof}

\section{Dense side}
The proof for the dense side is based on Radon's theorem. Recall that Radon's theorem, originally proved in \cite{Radon}, states that any set of $d + 2$ points in $\R^d$ can be partitioned into two sets whose convex hulls intersect and if the point set is generic the Radon partition is unique. Thus if we have a linear embedding of $X \sim X(n; n^{-\alphavector})$ into $\R^{2d}$ then for any set of $2d + 2$ vertices we get from Radon's theorem a pair of simplices on those vertices that must be excluded from $X$. Given an embedding of $n$ points in $\R^{2d}$ and a simplicial complex $X$ on $[n]$ we say that $S \in \binom{[n]}{2d + 2}$ is a \emph{Radon match} if in the Radon partition on $S$ into two subsets $S_1$ and $S_2$, the simplex on $S_1$ and the simplex on $S_2$ are both present in $X$. If $X$ has a Radon match for every embedding of the vertices of $X$ into $\R^{2d}$ then $X$ is not linearly embeddable in $\R^{2d}$. Our strategy is to show that with high probability when $X \sim X(n; n^{-\alphavector})$ on the dense side of the claimed phase transition every embedding of the vertices has at least one Radon match. We also will assume throughout that all embeddings have vertices placed generically in $\R^{2d}$ but it is clear that if there is a linear embedding of $X$ into $\R^{2d}$ then by a sufficiently small perturbation we have a generic embedding.

Turning the strategy into a proof requires three ingredients. First we need some control on which type of splits of our $(2d + 2)$-sets we see. If, for example, they were all split as $2d + 1$ vertices in convex position surrounding a single vertex in the interior, we would not expect to see any Radon partitions just beyond the phase transition. Second, for our purposes any two embeddings of the $n$ vertices in $\R^d$ can be regarded as equivalent if they induce the same Radon partition on all $2(d + 2)$-subsets. Therefore we need an upper bound on how many nonequivalent embeddings there are to consider. Third we need an upper bound on the probability that a particular embedding has no Radon match and this upper bound has to go to zero faster than the number of nonequivalent embeddings is going to infinity so that we can take a union bound.

For the first point we use the following stronger version of the classical van Kampen--Flores Theorem.
\begin{theorem}[Theorem 6.6 of \cite{BFZ}]\label{BFZtheorem}
Let $d \geq 1$. Then for every embedding of $d + 3$ points in $\R^d$ there exists two disjoint sets $S_1$ and $S_2$ with $|S_1| = \lfloor (d + 2)/2 \rfloor$ and $|S_2| = \lceil (d + 2)/2 \rceil$ so that the convex hull of $S_1$ and $S_2$ intersect.
\end{theorem}
As a corollary to this we have the following regarding an embedding of $n$ points in $\R^{d}$
\begin{corollary}\label{EvenSplitsCorollary}
For any embedding of $n$ points in $\R^d$ there are at least
\[\frac{1}{d + 3} \binom{n}{d + 2}\]
$(d + 2)$-subsets of $[n]$ that split into two sets $S_1$ and $S_2$ with $|S_1| = \lfloor (d + 2)/2 \rfloor$ and $|S_2| = \lceil (d + 2)/2 \rceil$ so that the convex hull of $S_1$ and $S_2$ intersect.
\end{corollary}
\begin{proof}
By Theorem \ref{BFZtheorem} for any embedding of $n$ points in $\R^d$ every subset of size $d + 3$ contributes at least one pair $(S_1, S_2)$ satisfying the required condition. Such a pair has $|S_1 \cup S_2| = d + 2$, and each belongs to at most $n - (d + 2)$ $(d + 3)$-element subsets. Thus the number of such pairs $(S_1, S_2)$ is at least
\[\frac{1}{n - (d + 2)}\binom{n}{d + 3} = \frac{1}{d + 3} \binom{n}{d + 2}.\]
\end{proof}

We now turn our attention to bounding the number of embeddings of the vertices we have to consider. We consider two general position embeddings $f, g: [n] \rightarrow \R^d$ to be equivalent if for every $S \in \binom{[n]}{d}$ the Radon partition of $S$ under the image of $f$ is the same as the Radon partition of $S$ under the image of $g$. 

For this we are really counting labeled, simple \emph{order types} of a configuration of $n$ points in $\R^d$. Recall that the order type of a labeled set of points $p_1, ..., p_n$ in $\R^d$ is the $\binom{n}{d + 1}$-tuple of signs of determinants of $(d + 1) \times (d +1)$ submatrices of 
\[\begin{pmatrix} p_1 & \cdots & p_n \\ 1 & \cdots & 1 \end{pmatrix}.\]

That is the order type is 
\[\left( \text{sgn} \det \begin{pmatrix} p_{i_1} & \cdots & p_{i_{d + 1}} \\ 1 & \cdots & 1 \end{pmatrix} \right)_{1 \leq i_1 < \cdots < i_{d + 1} \leq n}\]

By construction then for any set of $d + 2$ points in general position the order type of those $d + 2$ points determines the Radon partition. Indeed for $p_1$, ..., $p_{d + 1}$, $p_{d + 2}$ we have $p_1$, ..., $p_{d + 1}$ are vertices of a simplex and for any $i$, the order type (either $-1$ or $1$) of $p_1,.., p_{i - 1}, p_{i + 1}, .., p_{d + 1}, p_{d + 2}$ determines on which side of the hyperplane through $p_1,.., p_{i - 1}, p_{i + 1}, .., p_{d + 1}$, $p_{d + 2}$ sits. Thus the order type of $p_1, ..., p_{d + 1}, p_{d + 2}$ determines which region of the hyperplane arrangement on $p_1, ..., p_{d + 1}$ contains $p_{d + 2}$. This region uniquely determines the Radon partition on $p_1, .., p_{d + 2}$. For more background on order types see the survey of Goodman and Pollack \cite{GoodmanPollackSurvey}. By simple order types we mean order types where none of the determinant signs are zero, these are all the order types that have to be considered for generic point sets.

The following enumeration result of Goodman and Pollack on order types will be sufficient for our proof. There is also a stronger result for all labeled order types, not just simple ones, due to \cite{AlonOrderTypes}
\begin{theorem}\cite[Theorem 1]{GoodmanPollack}\label{GoodmanPollackTheorem}
The number of simple order types for $n$ labeled points in general position in $\R^d$ is at most 
\[n^{d(d + 1) n}.\]
\end{theorem}

We note that the number of order types is a technically a refinement of the number of embeddings which induce the same Radon partitions on every $(d + 2)$-subset of $n$ points. If we have only $d + 2$ points in general position we have a distinct Radon partition, $(A, B)$, but we can change the order type by permuting the points relabeling the points in $A$ and the points in $B$. However the direction that we need is that the order type determines the Radon partition discussed above. So we have the following lemma as an immediate corollary to Theorem \ref{GoodmanPollackTheorem}.

\begin{lemma}\label{NumberOfEmbeddings}
For $d \geq 1$ the number of nonequivalent, general position embeddings of $n$ vertices in $\R^d$ is at most
\[n^{d(d + 1) n}.\]
\end{lemma}
For the last piece of the proof for the dense side of the phase transition we bound the probability that for any fixed embedding of the vertices of $X$ we have no Radon matches.
\begin{lemma}\label{ConcentrationInequality}
Fix $d \geq 1$ and let $\pi: [n] \rightarrow \R^{2d}$ be a general position embedding for $n$ vertices in $\R^{2d}$. For $d + 1 - \alphavector \cdot v_{d + 1} > 1$ the probability that $X \sim X(n; n^{-\alphavector})$ has no Radon matches under the mapping into $\R^{2d}$ induced by $\pi$ is at most
\[\exp(-n^{1 + \epsilon})\]
for some $\epsilon > 0$ that depends only on $\alphavector$ and $d$.
\end{lemma}
For the proof we will make use of the following form of Janson's equality from Theorem 23.13 of \cite{FriezeRandomGraphs}:
\begin{theorem}[Janson's Inequality]
Let $R$ be a random subset of $[N]$ such that for each $s \in [N]$, $q_s \in (0, 1)$ denotes the probability that $s \in N$ is included in $R$. Let $D_1, ..., D_n$ be a family of $n$ subsets of $R$. Suppose that $S_n$ is a sum of indicator random variables $S_n = I_1 + \cdots + I_n$ where $I_i$ is the indicator for the event $D_i$. Write $i \sim j$ is $D_i \cap D_j \neq \emptyset$ let 
\[\overline{\Delta} = \sum_{\{i, j\}: i \sim j} \mathbb{E}(I_iI_j)\]
then 
\[\Pr(S_n =0) \leq \exp \left( \frac{-(\mathbb{E}(S_n))^2}{2\overline{\Delta}} \right).\]
\end{theorem}
In our case $N$ is all the faces of the simplex on $n$ vertices other than the vertices and $R$ is the random complex in $X(n; n^{-\alphavector})$. We might hesitate for a moment regarding the independence assumption as the faces are not included independently in $X$. However we can instead associate to $\sigma$ in the simplex on $n$ vertices $q_{\sigma} = n^{-\alpha_{|\sigma| - 1}}$, in this way each face is set to be either on or off independently and then the complex $X$ consists of those faces that are switched on and have all their subfaces switched on as well. Each $D_i$ will be a Radon match which corresponds to some collection of faces all switched on independently.
\begin{proof}[Proof of Lemma \ref{ConcentrationInequality}]
Let $Y$ be the random variable counting the number of evenly split Radon matches in $X \sim X(n; n^{-\alphavector})$ under $\pi$, i.e. Radon matches for Radon partitions with $d + 1$ vertices in each part. By linearity of expectation and Corollary \ref{EvenSplitsCorollary} we have
\[\E(Y) \geq \frac{1}{(2d + 3)(2d + 2)^{2d + 2}} n^{2d + 2} n^{-2\alphavector \cdot v_{d + 1}}  = \Theta(n^{2((d + 1) - \alphavector\cdot v_{d + 1})}).\]
Under our assumptions then $\E(Y) \rightarrow \infty$. Now we use Janson's inequality to bound the probability that $Y = 0$. Observe that $Y$ is a sum of indicator random variables $\textbf{1}_{(A, B)}$ where $A$, $B$ is an evenly split Radon partition of $A \cup B$ coming from $\pi$, with $A$ the lexicographically smaller of the two sets with respect to some ordering, and $\textbf{1}_{(A, B)}$ is the indicator random variable for the event that the simplex on $A$ and the simplex on $B$ are both included in $X$. If $(A, B)$ and $(A', B')$ don't share any edges, i.e. if $|A \cap B'|$, $|A \cap A'|$, $|B \cap A'|$, and $|B \cap B'|$ are all of size at most 1, then $\textbf{1}_{(A, B)}$ and $\textbf{1}_{(A', B')}$ are independent. To apply Janson's inequality then we let 
\[\overline{\Delta} = \sum_{\{(A, B), (A', B') \mid (A \cup B), (A' \cup B') \text{ share an edge}\}} \E(\textbf{1}_{(A, B)} \textbf{1}_{(A', B')}),\]
and we compute an upper bound on $\overline{\Delta}$. If $(A, B)$ and $(A', B')$, with $A \cap B$ and $A' \cap B'$ both empty,  share at least an edge then the intersection of the two pairs is a disjoint union of up to four nonempty sets $A \cap B'$, $A \cap A'$, $B \cap A'$, and $B \cap B'$. Let $m_1, m_2, m_3, m_4$ denote the sizes of these sets respectively, note that at least one of the $m_i$'s is at least 2 and all of them are at most $d + 1$. Given $m_1$, $m_2$, $m_3$, and $m_4$ the probability that $A, B, A', B'$ are all included in $X$ is
\[\E(\textbf{1}_{(A, B)} \textbf{1}_{(A', B')}) = n^{- 4\alphavector \cdot v_{d + 1} - (\sum_{i = 1}^4 - \alphavector \cdot v_{m_i})}\]
as the simplices in the intersection are counted exactly twice.

Now for each choice of $m_1, m_2, m_3, m_4$ the number of way to pick for sets of $d + 1$ vertices $A$, $B$, $A'$, and $B'$ so that $|A \cap B| = |A' \cap B'| = 0$ so that $|A \cap B'| = m_1$, $|A \cap A'| = m_2$, $|B \cap A'| = m_3$, and $|B \cap B'| = m_4$ is at most

\[O(n^{4(d + 1) - m_1 - m_2 - m_3 - m_4}).\]

It follows that 
\[\overline{\Delta} = O \left( \max_{m_1, m_2, m_3, m_4} n^{4((d + 1) -  \alphavector \cdot v_{d + 1}) - \sum_{i = 1}^4 (m_i - \alphavector \cdot v_{m_i})} \right)\]
where the maximum is taken over all choices of $m_1, m_2, m_3, m_4$ so that each is at most $d + 1$ and at least one of them is more than 1.
%%%Stopped here
By Janson's inequality then,
\[\Pr(Y = 0) \leq \exp\left(-\Theta\left(n^{\min_{m_1, m_2, m_3, m_4} \{\sum_{i = 1}^4 (m_i - \alphavector \cdot v_{m_i})\}} \right)\right).\]
We now need to verify that 
\[\min_{m_1, m_2, m_3, m_4} \left\{ \sum_{i = 1}^4 (m_i - \alphavector \cdot v_{m_i}) \right\} > 1\]
This follows though since at least one of the $m_i$'s is always at least 2 and they are all always at most $d + 1$. So by the second part of Lemma \ref{fvectortheorem} for all allowable choices of $m_1, m_2, m_3, m_4$,  $\sum_{i = 1}^4 (m_i - \alphavector \cdot v_{m_i}) > 1$. As there is a constant bound on the possible values for $m_1, ..., m_4$ for $d$ fixed, we have that there is some $\epsilon$ depending on $d$ so that 
\[\Pr(Y = 0) \leq \exp(n^{1 + \epsilon}).\]
\end{proof}

\begin{proof}[Proof of dense side of Theorem \ref{maintheorem}]
Set $d \geq 1$. By Lemma \ref{ConcentrationInequality} there is $\epsilon > 0$ so that for any embedding $\pi$ of the vertices of $X \sim X(n; n^{-\alphavector})$ into $\R^{2d}$, the probability that $X$ has no Radon matches with $\pi$ is at most $\exp(-n^{1 + \epsilon})$. Taking a union bound over all nonequivalent embeddings of the vertices from Lemma \ref{NumberOfEmbeddings} we have that the probability that there is an embedding of the vertices with no Radon match is at most 
\[n^{2d(2d + 1) n} e^{-n^{1 + \epsilon}}= o(1).\]
So with high probability every embedding of $X$ into $\R^{2d}$ has a Radon match, therefore $X$ is not linearly embeddable in $\R^{2d}$.
\end{proof}

\section{Concluding remarks}
Theorem \ref{maintheorem} establishes a threshold result for linear embeddings in even dimensions. Two important special cases are for the Linial--Meshulam--Wallach model where we see that (up to the right exponent) $p = n^{-d}$ is the threshold for $Y \sim Y_d(n, p)$ to be embeddable in $\R^{2d}$ and for the clique complex model where we see that $p = n^{-2/(d + 1)}$ is the threshold for embeddability of $X \sim X(n, p)$ in $\R^{2d}$. Up to the exponent, this matches the threshold for $X$ to be at least $(d + 1)$-dimensional even though we found obstructions to embeddability already in the $d$-skeleton.

PL-embeddings of random complexes have also been studied in the past, and together with a result of Wagner, Theorem \ref{maintheorem} gives a new, probabilistic proof of the known result that PL-embeddablility is a strictly weaker notion than linearly-embeddability for complexes of dimension larger than 1. Let $Y_d(n, p)$ denote the $d$-dimensional Linial--Meshulam--Wallach model. Wagner's result on PL-embeddability is the following
\begin{theorem}[Theorem 2 of \cite{Wagner}]
The threshold for PL-embeddability of $Y \sim Y_d(n, p)$ into $\R^{2d}$ is at $p = \Theta(1/n)$.
\end{theorem}
From this and Theorem \ref{maintheorem} the following is immediate.
\begin{theorem}
For $1 < \alpha < d$ with high probability $Y \sim Y_d(n, n^{-\alpha})$ is PL-embeddable in $\R^{2d}$, but not linearly embeddable in $\R^{2d}$.
\end{theorem}

Perhaps the most natural open question is to establish linear embedding thresholds for odd dimensions. An earlier version of this paper claimed such a threshold, however upon writing this revision an error was found for the proof in the odd dimensional case. For showing nonembeddability in $\R^{2d - 1}$ for sparser complexes than those that fail to embed in $\R^{2d}$, some new ideas would be required. To explain why we'll examine the case of embedding the random clique complex in $\R^3$. 

For $2/3 < \alpha < 1$ we know that $X(n, n^{-\alpha})$ embeds in $\R^4$ but not in $\R^2$. If we take $\alpha = 1 - \delta$ for some small $\delta$ and try to use the same proof based on Radon matches to rule out embeddability into $\R^3$, we could still show that on average a fixed placement of the vertices in $\R^3$ will have $\Theta(n^{5-4\alpha})$ Radon matches, so that part of the proof would still be consistent with the even dimensional case. The problem, however, is that we cannot get an upper bound on the probability of no Radon matches of $\exp(-n^{1 + \epsilon})$ in this case, and this was used in the even dimensional case to take a union bound over the approximately $n^n$ order types. The reason we can't prove such a bound is that if our random complex happens to have no triangles at all for $\alpha = 1 - \delta$ then the complex is 1-dimensional and so it will embed in $\R^3$ generically. The average number of triangles is $n^{3 - 3 \alpha} = n^{-3 \delta}$, so the probability there are no triangles should be roughly $\exp(-n^{3\delta})$ (there would be some things to check here since the triangles aren't included independently, but this is just a sketch of the idea). So there is a lower bound that exceed $\exp(-n^{1 + \epsilon})$ until $\delta > 1/3$, but at that point we are at the nonembeddability threshold for $\R^4$ anyway.

Other remaining open problems include establishing sharp threshold results for linear embeddability in even dimensions; the main result here establishes thresholds for linear embeddability only up to the right exponents. Additionally, now that we know that the Linial--Meshulam--Wallach model has different thresholds for PL-embeddability and for linear embeddability, it would be interesting then to determine PL-embeddability thresholds for entire the multiparameter model. 

\section*{Acknowledgments}
The author thanks Florian Frick for discussions leading to the central question answered in this paper and for helpful comments on an early draft and Boris Bukh for pointing out an error in a previous version.
\bibliography{ResearchBibliography}
\bibliographystyle{amsplain}
\end{document}